\documentclass[12pt]{amsart}
\usepackage{amssymb}
\usepackage{booktabs}
\usepackage[numbers,sort&compress]{natbib}
\usepackage{color}
\usepackage{paralist}
\usepackage{hyperref}
\newtheorem {theorem}{Theorem}[section]
\newtheorem {proposition}{Proposition}[section]
\newtheorem {corollary}{Corollary}[section]
\newtheorem {lemma}{Lemma}[section]
\newtheorem {definition}{Definition}[section]
\newtheorem {remark}{Remark}[section]
\newtheorem {assumption}{Assumption}[section]

\newcommand{\M}{{\rm \bf M}}
\newcommand{\R}{{\mathbb R}}
\newcommand{\N}{\mathbb{N}}
\newcommand{\ba}{{\boldsymbol \alpha}}
\newcommand{\bb}{{\boldsymbol \beta}}
\newcommand{\bx}{\textrm{\bf{x}}}
\newcommand{\bz}{\textrm{\bf{z}}}
\newcommand{\p}{\textrm{\bf{p}}}
\newcommand{\q}{\textrm{\bf{q}}}
\newcommand{\by}{\textrm{\bf{y}}}
\newcommand{\bze}{\boldsymbol 0}
\newcommand{\bfd}{\textrm{\bf{d}}}
\newcommand{\dd}{d_{\min}}

\title[A bounded degree SDSOS hierarchy for solving SPLD polynomial programs]
{A bounded degree SDSOS hierarchy for solving SPLD polynomial programs}

\author{Chengmiao Yang}
\address[Chengmiao Yang]{Academy for Advanced Interdisciplinary Studies, Northeast Normal University, Changchun 130024, Jilin Province, China}
\email{cmyang@nenu.edu.cn}

\author{Liguo Jiao}
\address[Liguo Jiao]{Academy for Advanced Interdisciplinary Studies, Northeast Normal University, Changchun 130024, Jilin Province, China}%; Shanghai Zhangjiang Institute of Mathematics, Shanghai 201203, China.}
\email{jiaolg356@nenu.edu.cn; hanchezi@163.com}

\author{Jae Hyoung Lee$^{\dag}$}
\address[Jae Hyoung Lee]{Department of Applied Mathematics, Pukyong National University, Busan 48513, Korea.}
\email{mc7558@pknu.ac.kr; mc7558@naver.com}

\thanks{$^{\dag}$Corresponding Author}



\date{\today}

\begin{document}

\begin{abstract}
A polynomial optimization problem is said to have an SPLD structure if the polynomials defining its objective and constraints admit decompositions into separable and nonseparable parts, with the maximum degree of the separable parts exceeding that of the nonseparable parts.
In this paper, we propose a bounded-degree SDSOS hierarchy for globally solving a class of such problems.
We establish the asymptotic convergence of this hierarchy over compact feasible sets and provide a sufficient condition for its finite convergence.
Numerical experiments compare the proposed hierarchy with the BSDSOS hierarchy (Chuong et al. in J. Global Optim., 75(4):885--919, 2019).
\end{abstract}

\maketitle

\section{Introduction}\label{sect:1}

Consider the following polynomial optimization problem (POP for short),
\begin{align}\label{POP}
f^*\ :=\ \min\limits_{\bx\in\R^n} \ \left\{ f_0(\bx) : f_i(\bx)\geq 0, \ i=1,\ldots,m \right\}, \tag{${\rm POP}$}
\end{align}
where $f_i\colon\R^n\to\R,$ $i=0,1,\ldots,m,$ are real multivariate polynomial functions.
POPs have many applications in operations research, control, engineering, economics and statistics~\cite{Aittomaki2009,Bai2008,Curmei2025,Lasserre2015,Luo2010,Singer2011}.
The well-known moment-SOS hierarchy \cite{Lasserre2001} (where SOS stands for {\em sum of squares}) is a globally solving method based on certain specific positivity certificates of real algebraic geometry \cite{Putinar1993}.
It then results in solving a hierarchy of semidefinite program (SDP for short) relaxations of the problem~\eqref{POP}, whose associated monotone sequence of optimal values provides lower bounds and converges to $f^*$.
Later on, the convergence is proved to be finite generically \cite{Nie2014}.
However, in view of the present status of SDP solvers, the moment-SOS hierarchy does not scale well and is therefore limited to problems of modest size.

In order to overcome the issue of scalability, several research directions have been developed.
For instance, sparsity or symmetry exploiting variants of the moment-SOS hierarchy are taken into account to solve substantially large-scale problems~\cite{GuoWang2025,Waki2006,Wang2021,Wang2022}.
Another possibility is to define hierarchies of relaxations of the problem~\eqref{POP} based on other alternative positivity certificates (e.g., Krivine--Stengle's positivity certificate \cite{Krivine1964,Stengle1974}).
Such certificates are in general weaker but their implementation is easier as it results in LP-relaxations (e.g., DSOS \cite{Ahmadi2019}), second-order cone relaxations (e.g., SDSOS \cite{Ahmadi2019}), or cheaper SDP-relaxations (e.g., BSOS \cite{Lasserre2017}).

\subsection{On types of bounded degree SOS hierarchies}
Among many others, we are interested in SDSOS~\cite{Ahmadi2019} and BSOS~\cite{Lasserre2017}.
Indeed, as an extension of BSOS, a bounded degree mixed SDP-SOCP hierarchy for solving the problem~\eqref{POP} has been proposed by Chuong et al.~\cite{Chuong2019}.
Here, we review some types of bounded degree SOS hierarchies for globally solving the problem~\eqref{POP}.
Such hierarchies are based on the well-known Krivine--Stengle's certificate of positivity in real algebraic geometry~\cite{Krivine1964,Stengle1974}.
\begin{enumerate}[\upshape (i)]
\item {\bf BSOS hierarchy~\cite{Lasserre2017}}. The BSOS hierarchy was initially proposed by Lasserre et al.~\cite{Lasserre2017} (see also \cite{Weisser2018} for its sparse version) to address the computational burden of the semidefinite relaxations and the lack of finite convergence of the LP-relaxations~\cite{Lasserre2002}.
The key idea is to add SOS constraints with fixed degree to the LP hierarchy, thereby predefining the size of the semidefinite matrix for each relaxation in the hierarchy.
The BSOS hierarchy approximates the optimum value by solving a series of SDP problems as follows,
\begin{equation*}
q_{d}^{k}=\sup _{\substack{\mu\in \mathbb{R}, \\ c_{\p, \q} \geq 0}}
\left\{\mu : f_{0}-\mu-\sum_{(\p, \q) \in \mathbb{N}_{d}^{2 m}} c_{\p, \q}\left(\prod_{i=1}^{m} f_{i}^{p_{i}}\left(1-f_{i}\right)^{q_{i}}\right) \in \Sigma[\bx]_{k}\right\},
\end{equation*}
where $\mathbb{N}_{d}^{2 m}=\left\{(\p, \q) \in \mathbb{N}_0^{2 m}:|\p|+|\q| \leq d\right\}$ and $\N_0$ denotes the sets of nonnegative integers. 
In contrast to the LP hierarchy, for convex problems defined with convex quadratic polynomials or SOS-convex polynomials of degree at most $2k$, finite convergence occurs at the first step in the BSOS hierarchy.
Furthermore, similar to the standard SOS hierarchy, a sufficient condition also exists for the finite convergence of this hierarchy.

\item {\bf Bounded degree mixed SDP-SOCP hierarchy~\cite{Chuong2019}}.
Recently, Chuong et al.~\cite{Chuong2019} developed a new bounded degree mixed SDP-SOCP hierarchy for globally solving the problem~\eqref{POP}.
By dividing the variables into two groups,  $X_{1}=\left(x_{1}, \ldots, x_{n_{1}}\right)$ and  $X_{2}=\left(x_{n_{1}+1}, \ldots, x_{n}\right),$ two corresponding types of relaxation constraints are defined:  $\sigma_{1} \in {\rm SOS}_{k}\left[X_{1}\right]$ and $\sigma_{2} \in {\rm SDSOS}_{k}\left[X_{2}\right],$ where ${\rm SOS}_{k}\left[X_{1}\right]$ and ${\rm SDSOS}_{k}\left[X_{2}\right]$ denote the set of all SOS polynomials in variable $X_{1}$ with degree at most $k$ and the set of all scaled diagonally dominant SOS polynomials in variable $X_{2}$ with degree at most $k,$ respectively.
The size of both types of constraints is fixed by a preset parameter $k$ and is independent of the relaxation level $d.$
The $d$-th relaxation is as follows,
\begin{equation*}
\sup _{\substack{\mu \in \mathbb{R}, c_{\mathbf{p}, \mathbf{q}} \geq 0, \\ \sigma_{1} \in \mathbf{SOS}_{k}\left[X_{1}\right], \sigma_{2} \in \mathbf{SDSOS}_{k}\left[X_{2}\right]}}\left\{\mu : f_{0}-\mu-\sum_{(\p, \q) \in \mathbb{N}_{d}^{2 m}} c_{\mathbf{p}, \mathbf{q}} \left(\prod_{i=1}^{m} f_{i}^{p_{i}}\left(1-f_{i}\right)^{q_{i}}\right)=\sigma_{1}+\sigma_{2}\right\}.
\end{equation*}
They also proved that, under certain technical assumptions, the optimal values of the primal and dual relaxations both converge to the global optimum of the original problem~\eqref{POP}.

\item {\bf BSOS-SPLD hierarchy~\cite{Jiao2026}}.
For a class of optimization problems with SPLD polynomials (see Definition~\ref{SPLD_poly}), Jiao et al.~\cite{Jiao2026} proposed a variant of the bounded degree SOS hierarchy (called BSOS-SPLD) for solving it.
It decomposes the nonnegative terms into a sum of low degree SOS polynomials and univariate SOS polynomials.
For fixed $\mathbf{r}:=\left(r_{1}, \ldots, r_{n}\right) \in \mathbb{N}^{n}$ and $k \in \mathbb{N},$ the $d$-th hierarchy is as follows,
\begin{equation*}
\sup _{\substack{\mu \in \mathbb{R}, c_{\mathbf{p}, \mathbf{q}} \geq 0, \\ \sigma \in \Sigma[\mathbf{x}]_{k}, \sigma_{j} \in \Sigma\left[x_{j}\right]_{r_{j}}}}\left\{\mu : f_{0}-\mu-\sum_{(\mathbf{p}, \mathbf{q}) \in \mathbb{N}_{d}^{2m}} c_{\mathbf{p}, \mathbf{q}} \left(\prod_{i=1}^{m} f_{i}^{p_{i}}\left(1-f_{i}\right)^{q_{i}}\right)=\sigma+\sum_{j=1}^{n} \sigma_{j}\right\} ,
\end{equation*}
where $\mathbb{N}_{d}^{2 m}=\left\{(\p, \q) \in \mathbb{N}^{2 m}:|\p|+|\q| \leq d\right\}.$
Compared to the standard BSOS hierarchy~\cite{Lasserre2017}, a notable feature of the BSOS-SPLD hierarchy is that the maximum size of the semidefinite constraints in its relaxations may be significantly smaller.
\end{enumerate}

\subsection{Motivation}
As we have mentioned before, very recently, Jiao et al.~\cite{Jiao2026} defined the so-called SPLD polynomial (see Definition~\ref{SPLD_poly}), which is a new class of polynomials extending the concept of SPQ (Separable Plus Quadratic) polynomials proposed by Ahmadi et al.~\cite{Ahmadi2022}.
Moreover, they proposed a bounded degree SOS hierarchy to solve the optimization problem with SPLD polynomials \cite{Jiao2026}.

The present work can be seen as a follow-up to~\cite{Jiao2026}.
Unlike~\cite{Jiao2026}, however, we do not require every polynomial in the problem to be an SPLD polynomial.
Instead, we consider the broader class of SPLD polynomial programs defined in Section~\ref{sect:3} and propose a bounded degree SDSOS hierarchy for globally solving them.
The idea is inspired by Ahmadi and Majumdar~\cite{Ahmadi2019} and Chuong et al.~\cite{Chuong2019}.

\subsection{Contributions}
We make the following contributions to global polynomial optimization.
\begin{itemize}
\item We propose a bounded degree SDSOS hierarchy, called the BSDSOS-SPLD hierarchy, for globally solving a class of SPLD polynomial programs.
\item We provide a sufficient condition for finite convergence and show how an optimal solution can be extracted.
\end{itemize}

The remainder of the paper is organized as follows.
Section~\ref{sect:2} presents some notation and preliminaries on polynomials.
The main results are given in Section~\ref{sect:3}.
Conclusions are given in Section~\ref{sect:conclusion}.

\section{Preliminaries}\label{sect:2}

We begin this section by fixing some notation and preliminaries.
The set of natural numbers is denoted by $\N$.
Let $\R^n$ denote the $n$-dimensional Euclidean space, and let $\bx:=(x_1,\ldots,x_n)$ denote an element of $\R^n$.
The nonnegative orthant of $\R^n$ is denoted by
$\R_{+}^n:=\left\{\bx \in \R^n : x_i \geq 0,\ i = 1,\ldots, n \right\}.$

Let $\mathcal{S}^n$ be the set of $n\times n$ symmetric matrices with the usual inner product $\langle A, B\rangle := {\rm trace}(AB)$ for $A,B\in \mathcal{S}^n.$
Let $X\in \mathcal{S}^n.$
We say that $X$ is {\em positive semidefinite}, denoted by $X\succeq 0,$ if $\bz^TX\bz\geq 0$ for all $\bz\in\R^n,$ and denote by $\mathcal{S}_+^n$ the set of $n\times n$ symmetric positive semidefinite matrices.
Moreover, $X$ is said to be {\em diagonally dominant} (``dd" in short) if $X_{ii}\geq \sum_{j\neq i}|X_{ij}|$ for all $i=1, \ldots, n,$ and the set of $n\times n$ symmetric diagonally dominant matrices is denoted by $\mathcal{S}_{{\rm dd}}^n$.
More generally, $X$ is said to be {\em scaled diagonally dominant} (``sdd" in short) if there exists a diagonal matrix $D$ with all positive diagonal elements such that $D^TXD$ is a diagonally dominant matrix, and the set of $n\times n$ symmetric scaled diagonally dominant matrices is denoted by $\mathcal{S}_{{\rm sdd}}^n$.
The relationship among the above-mentioned matrices is $\mathcal{S}_{{\rm dd}}^n \subseteq \mathcal{S}_{{\rm sdd}}^n \subseteq \mathcal{S}_+^n,$ and the inclusions may be strict (see, e.g., \cite{Ahmadi2019}).

The ring of all real polynomials in the variable $\bx$ is denoted by $\R[\bx].$
Moreover, the vector space of all real polynomials in the variable $\bx$ with degree at most $d$ is denoted by $\R[\bx]_d.$
The degree of a polynomial $f$ is denoted by $\deg (f).$
For a multi-index $\ba:=(\alpha_1,\ldots,\alpha_n)\in\N_0^n,$ let us denote $|\ba|:=\sum_{i=1}^n\alpha_i$ and $\N^n_d:= \{\ba \in \N_0^n \colon  |\ba|\leq d\}.$
The notation $\bx^\ba$ denotes the monomial $x_1^{\alpha_1}\cdots x_n^{\alpha_n}.$
The canonical basis of $\R[\bx]_d$ is denoted by
\begin{equation}\label{cano_basis}
\lceil \bx\rceil_d:=(\bx^\ba)_{\ba\in\N^n_{d}}=(1,x_1,\ldots,x_n,x_1^2,x_1x_2,\ldots,x_n^2,\ldots,x_1^d,\ldots,x_n^d)^T,
\end{equation}
which has dimension $s(n,d):=\binom{n+d}{d}.$

We say that a real polynomial $f$ is sum-of-squares (SOS for short)
if there exist real polynomials $p_j,$ $j = 1,\ldots,q,$ such that $f =\sum_{j=1}^{q}p_{j}^2.$
The set consisting of all SOS real polynomials in variable $\bx$ is denoted by $\Sigma[\bx]$, clearly it holds that $\Sigma[\bx]\subset \R[\bx].$
In addition, the set consisting of all SOS real polynomials with degree at most $2d$ is denoted by
$\Sigma[\bx]_{d}$, similarly one has $\Sigma[\bx]_{d}\subset\R[\bx]_{2d}.$
For $f\in\R[\bx]_d,$ we denote by $f_\ba$ $(\ba\in\N^n_d)$ the coefficients of $f.$

Given an $s(n, 2d)$-vector $\by:= (y_\ba)_{\ba\in\N^n_{2d}}$ with $y_{\bze}=1,$ let $\M_{d}(\by)$ be the moment matrix of dimension $s(n, d),$ with rows and columns labeled by \eqref{cano_basis} in the sense that
\begin{equation*}
\M_d(\by)(\ba,\bb)=y_{\ba+\bb}, \ \forall \ba,\bb\in\N^n_d.
\end{equation*}
Namely, $\M_d(\by)=\sum_{\ba\in\N^n_{2d}}y_\ba (\lceil \bx\rceil_d\lceil \bx\rceil_d^T)_\ba.$

\medskip
The following proposition shows that SOS decomposability can be checked by solving an SDP.
\begin{proposition}\label{proposition1}{\rm \cite{Lasserre2010,Lasserre2015book}}
A polynomial $f\in\R[\bx]_{2d}$ has an SOS decomposition if and only if there exists $Q\in \mathcal{S}^{s(n,d)}_+$ such that
\begin{equation*}
\langle (\lceil \bx\rceil_d\lceil \bx\rceil_d^T)_\ba ,Q\rangle=f_\ba, \ \forall \ba\in\N_{2d}^n.
\end{equation*}
\end{proposition}

Very recently, as a subclass of SOS polynomials, a new notion called scaled diagonally dominant sums-of-squares (SDSOS for short) polynomial was introduced by Ahmadi and Majumdar~\cite{Ahmadi2019}; see the following definition.
\begin{definition}[SDSOS polynomial]\label{SDSOS}{\rm
A polynomial $f\in\R[\bx]_{2d}$ is called {\it SDSOS} if it can be written as
\begin{equation*}
f(\bx) = \sum_{i=1}^p\alpha_im^2_i(\bx) + \sum_{i,j=1}^p(\beta^+_{ij}m_i(\bx) + \gamma^+_{ij}m_j(\bx))^2 +\sum_{i,j=1}^p(\beta^-_{ij}m_i(\bx)-\gamma^-_{ij}m_j(\bx))^2,
\end{equation*}
for some monomials $m_{i}(x),$ $m_{j}(x),$ $p\in\N$ with $1\leq p \leq s(n,d)$ and some scalars $\alpha_i,$ $\beta^\pm_{ij},$ $\gamma^\pm_{ij}$ with $\alpha_i\geq0.$
We denote the set of SDSOS polynomials of degree $2d$ in $\bx$ by $\widetilde\Sigma[\bx]_{d}.$
}\end{definition}

Note that a polynomial $f\in\R[\bx]_{2d}$ is SDSOS if and only if it admits a representation as $f(\bx) = \lceil \bx\rceil_d^TQ \lceil \bx\rceil_d$ for some $s(n,d)\times s(n,d)$ scaled diagonally dominant matrix $Q$; see, e.g., \cite[Theorem~3.6]{Ahmadi2019}.
In other words, checking a given real polynomial is SDSOS or not can be done by solving a second-order cone programming (SOCP for short) problem.
Observe that any SDSOS polynomial is SOS, but the converse is not true in general~\cite{Ahmadi2019}.

\medskip
We close this section by recalling the following concept proposed by Jiao et al. \cite{Jiao2026}.
\begin{definition}[SPLD polynomial]\label{SPLD_poly}{\rm
A nonconstant polynomial $f: \R^n \to \R$ is said to be {\it separable plus lower degree} (SPLD for short) if
\begin{align*}
f(\bx) = s(\bx) + l(\bx),
\end{align*}
where $s(\bx)$ is a separable polynomial, i.e., $s(\bx):= \sum_{i = 1}^ns_i(x_i)$ for some univariate polynomials $s_i: \R \to \R,$ and $l(\bx)$ is a polynomial with its degree lower than the degree of $s(\bx).$
}\end{definition}

\section{The BSDSOS-SPLD hierarchy}\label{sect:3}

We consider the following polynomial optimization problem:
\begin{align}\label{P}
\inf_{\bx\in\R^n}\left\{f_0(\bx) \colon 0\leq f_i(\bx)\leq 1,\ i=1,\ldots,m\right\}. \tag{${\rm P}$}
\end{align}
Here, each $f_i\in\R[\bx]$, $i=0,1,\ldots,m$, is a nonconstant polynomial admitting a decomposition
\begin{equation*}
f_i(\bx)=s_i(\bx)+l_i(\bx), \ i=0,1,\ldots,m,
\end{equation*}
where $s_i$ is separable and $l_i$ contains the nonseparable terms.
Set
\begin{equation*}
D_s:=\max_{0\leq i\leq m}\deg(s_i) \quad\text{and}\quad D_l:=\max_{0\leq i\leq m}\deg(l_i).
\end{equation*}
We say that problem~\eqref{P} has an SPLD structure if
\begin{equation*}
D_s>D_l.
\end{equation*}
Such a problem is called an SPLD polynomial program.

The feasible set of problem~\eqref{P} is denoted by
\begin{align}\label{feasible-k}
\mathbf{K}:=\left\{\bx\in\R^n \colon 0\leq f_i(\bx)\leq 1,\ i=1,\ldots,m\right\}.
\end{align}

In this paper, we make the following assumptions on the feasible set $\mathbf{K}.$
\begin{assumption}\label{assumption1}
We assume that \begin{itemize}
\item[{\rm (i)}] $\mathbf{K}$ is a nonempty compact set;
\item[{\rm (ii)}] $\{1,f_1,\ldots,f_m\}$ generates the ring $\R[\bx];$
\end{itemize}
\end{assumption}

We next note that any nonempty compact feasible set can be represented in the form~\eqref{feasible-k}, and that Assumption~\ref{assumption1}(ii) can be ensured without loss of
generality.
Indeed, let
\begin{equation*}
\mathbf{K}':=\left\{x\in\mathbb{R}^n:g_i(x)\geq 0,\ i=1,\ldots,m\right\}
\end{equation*}
be nonempty and compact. For each $i=1,\ldots,m$, choose $M_i>0$ such that
\begin{equation*}
M_i>\max_{x\in \mathbf{K}'}g_i(x),
\end{equation*}
and define
\begin{equation*}
f_i(x):=\frac{g_i(x)}{M_i}.
\end{equation*}
Then
\begin{equation*}
\mathbf{K}'=\left\{x\in\mathbb{R}^n:0\leq f_i(x)\leq 1,\ i=1,\ldots,m\right\}.
\end{equation*}
Moreover, if the set $\mathbf{K}$ is nonempty and compact, then Assumption~(ii) can be ensured by adding suitably normalized redundant affine coordinate functions.

Now, we recall a well-known result (due to Krivine \cite{Krivine1964} and Stengle \cite{Stengle1974}) on the representation of polynomials that are positive on $\mathbf{K}$.
\begin{lemma}[Krivine--Stengle's positivstellensatz \rm{\cite{Lasserre2010,Krivine1964,Stengle1974}}]\label{KS_positivity}
Let $f_i,$ $i =0, 1,\ldots, m$ be real polynomials and
$\mathbf{K} := \{\bx \in\R^n \colon 0\leq f_i(\bx) \leq 1, \ i = 1,\ldots, m\}.$
Suppose that {\rm Assumption~\ref{assumption1}} holds.
If $f_0 > 0$ on $\mathbf{K},$ then there exist finitely many nonnegative coefficients $c_{\p,\q}$ such that
\begin{align*}
f_0 - \sum_{\p,\q\in \N^m_0}c_{\p,\q} \prod_{i=1}^m  f_i^{p_i} (1-f_i)^{q_i}\equiv0.
\end{align*}
\end{lemma}

\subsection{A bounded degree SDSOS hierarchy for SPLD polynomial programs}

In this subsection, we introduce a bounded degree SDSOS (BSDSOS) hierarchy for globally solving problem~\eqref{P}.
For convenience, write
\begin{equation*}
s_i(\bx):=\sum_{j=1}^n u_i^j(x_j), \ i=0,1,\ldots,m.
\end{equation*}
We use the following degree parameters:
\begin{itemize}
\item[$\bullet$] For each $i=0,1,\ldots,m$ and $j=1,\ldots,n$, let $d_i^j$ be the smallest nonnegative integer such that
\begin{equation*}
2d_i^j\geq \deg(u_i^j).
\end{equation*}
\item[$\bullet$] For each $i=0,1,\ldots,m$, let $r_i$ be the smallest nonnegative integer such that
\begin{equation*}
2r_i\geq \deg(l_i).
\end{equation*}
\item[$\bullet$] Choose fixed nonnegative integers $d_1,\ldots,d_n$ and $r$ satisfying
\begin{equation*}
d_j\geq \max_{0\leq i\leq m}d_i^j, \ j=1,\ldots,n, \quad \text{and} \quad r\geq \max_{0\leq i\leq m}r_i.
\end{equation*}
\item[$\bullet$] For each relaxation order $k\geq 1$, let $\dd$ be the smallest nonnegative integer such that
\begin{equation}\label{dmax}
2\dd\geq\max\left\{\deg(f_0),2\max_{1\leq j\leq n}d_j,2r,k\max_{1\leq i\leq m}\deg(f_i)\right\}.
\end{equation}
\end{itemize}

Now, for a given fixed integer $k\geq1,$ we consider the following problem,
\begin{align}\label{Pk}
\inf\limits_{\bx\in\R^n} \left\{ f_0(\bx) \colon  h_{\p,\q} (\bx) \geq 0, \ \forall (\p,\q)\in N_k \right\}, \tag{${\rm P}_k$}
\end{align}
where $N_k:=\left\{(\p,\q)\in\N^m_0\times\N^m_0 : |\p|+|\q|\leq k\right\}$ and $h_{\p,\q}\colon\R^n\to\R$ is the polynomial defined by
\begin{align*}
h_{\p,\q}(\bx):=\prod_{i=1}^m{f}_i(\bx)^{p_i}(1-{f}_i(\bx))^{q_i}
\end{align*}
for all $\bx\in \R^n.$
Indeed, the problem \eqref{Pk} is equivalent to \eqref{P} in the sense that \eqref{Pk} is formulated by adding redundant constraints, which are identically nonnegative over the feasible set of the problem \eqref{P}; see \cite[Section 2.3]{Lasserre2017} in detail.
Moreover, the Lagrangian dual problem associated with \eqref{Pk} is given by
\begin{align}\label{Dk}
\sup\limits_{\substack{\mu\in\R,c_{\p,\q}\geq0}}\left\{\mu : f_0(\bx) - \sum_{(\p,\q)\in N_k}c_{\p,\q}h_{\p,\q}(\bx) - \mu \geq 0, \ \forall \bx\in\R^n\right\}.\tag{${\rm D}_k$}
\end{align}
By weak duality, ${\rm val}\eqref{Dk}\leq{\rm val}\eqref{Pk} ={\rm val}\eqref{P}$.

For fixed $\bfd:=(d_1,\ldots, d_n) \in \N^n$ and $r \in \N,$ we consider the following $k$-th order hierarchy of SDSOS relaxations for the problem~\eqref{P},
\begin{align}\label{SDkdr}
\sup\limits_{\substack{\mu\in\R,c_{\p,\q}\geq0, \sigma\in\widetilde\Sigma[\bx]_{r}, \sigma_j\in\Sigma[x_j]_{d_j} }}\left\{ \mu \colon f_0-\sum_{(\p,\q)\in N_k}c_{\p,\q}h_{\p,\q}-\mu=\sigma+\sum_{j=1}^n\sigma_j\right\}.\tag{${\rm SD}_k^{\bfd,r}$}
\end{align}
We refer to this hierarchy as {\bf BSDSOS-SPLD}.

\begin{remark}[Comparison with the BSDSOS hierarchy \cite{Chuong2019}]
{\rm
Since $D_s>D_l$, terms of degree greater than $D_l$ can occur only in the separable polynomials $s_i$.
However, the condition $D_s>D_l$ does not require every polynomial $f_i$ to be an SPLD polynomial in the sense of Definition~\ref{SPLD_poly}.

In the BSDSOS hierarchy, the degree bound for the multivariate SDSOS polynomial must be fixed at a value no smaller than the maximum degree appearing in the polynomial optimization problem, including the degrees of the high-degree separable terms.
By contrast, the BSDSOS-SPLD hierarchy represents the separable parts using univariate SOS polynomials.
Hence, it is enough to choose $r$ so that $2r$ is no smaller than the maximum degree of the nonseparable parts.
Consequently, the BSDSOS-SPLD hierarchy may require a substantially smaller SOCP representation for the multivariate SDSOS polynomial.

However, the products $h_{\p,\q}$ do not generally preserve the SPLD structure of the original polynomials.
Therefore, the condition $D_s>D_l$ motivates the proposed decomposition but does not by itself guarantee exactness at a given relaxation order.
The coefficients $c_{\p,\q}$ combine the products
$h_{\p,\q}$, and the resulting residual polynomial is represented
by the fixed-degree SDSOS and univariate SOS components.
}
\end{remark}

As noted after Definition~\ref{SDSOS}, if $\sigma(\bx) \in \widetilde \Sigma[\bx]_{r},$ then $\sigma(\bx) = \lceil \bx\rceil_r^TX \lceil \bx\rceil_r$ for some $X \in  \mathcal{S}^{{s(n,r)}}_{{\rm sdd}}.$
This, along with Proposition~\ref{proposition1}, tells us that the problem~\eqref{SDkdr} can be reformulated as follows,
\begin{align}\label{problem0}
\sup\limits_{\substack{\mu,c_{\p,\q}, X, X_j}} \bigg\{ \mu\in\R :\,& \bigg(f_0-\sum_{(\p,\q)\in N_k}c_{\p,\q}h_{\p,\q}-\mu\bigg)_\ba=\langle B_\ba, X\rangle+\sum_{j=1}^n\langle B_{\ba}^j, X_j\rangle, \ \forall \ba\in \N^n_{2\dd}, \\
&c_{\p,\q}\geq0, \ (\p,\q)\in N_k, \ X\in \mathcal{S}^{{s(n,r)}}_{{\rm sdd}}, \  X_j\in \mathcal{S}^{d_j+1}_+, \  j=1,\ldots,n\bigg\}, \notag
\end{align}
where $\dd$ is defined as in \eqref{dmax}, and the symmetric matrices $B_\ba$ and each $B_{\ba}^j$ are defined, respectively, by
\begin{equation*}
B_\ba:=\left(\lceil \bx\rceil_r\lceil \bx\rceil_r^T\right)_\ba \quad \textrm{ and } \quad B_{\ba}^j:=\left(\lceil x_j\rceil_{d_j}\lceil x_j\rceil_{d_j}^T\right)_{\ba}
\end{equation*}
for all $\ba\in \N^n_{2\dd}.$

\begin{remark}\label{remark1}
{\rm Note from \cite[Lemma~3.8]{Ahmadi2019} that, for a matrix $X\in \mathcal{S}^{{s(n,r)}}_{{\rm sdd}}$, it can be expressed as
$ X=\sum_{1\leq i<j\leq s(n,r)}M^{ij},$
where each $M^{ij}$ is an $s(n,r)\times s(n,r)$ matrix with zeros everywhere except for four entries $M^{ij}_{ii},$  $M^{ij}_{ij},$ $M^{ij}_{ji},$ $M^{ij}_{jj},$ which form a $2\times 2$ matrix
\begin{align*}
\left(
  \begin{array}{cc}
    M^{ij}_{ii} & M^{ij}_{ij} \\
    M^{ij}_{ji} & M^{ij}_{jj} \\
  \end{array}
\right)\in \mathcal{S}_+^2,
\end{align*}
and it can be equivalently expressed by using SOCP (see, e.g., \cite{Alizadeh2003,Lobo1998}) as follows,
\begin{align*}
M^{ij}_{ii}+M^{ij}_{jj} \geq 0 \quad \textrm{and} \quad
\left\|\left(
  \begin{array}{c}
    2M^{ij}_{ij} \\
    M^{ij}_{ii} - M^{ij}_{jj} \\
  \end{array}
\right)\right\|\leq M^{ij}_{ii}+M^{ij}_{jj}.
\end{align*}
}\end{remark}

Thanks to Remark~\ref{remark1}, we now reformulate the problem~\eqref{problem0} into the following SDP-SOCP problem,
\begin{align*}
\sup\limits_{\substack{\mu,c_{\p,\q}, X, X_j,M^{ij}}} \bigg\{ \mu\in\R  : \, & \bigg(f_0-\sum_{(\p,\q)\in N_k}c_{\p,\q}h_{\p,\q}-\mu\bigg)_\ba=\langle B_\ba, X\rangle+\sum_{j=1}^n\langle B_{\ba}^j, X_j\rangle, \ \forall \ba\in \N^n_{2\dd},\\
& X_j\in \mathcal{S}^{d_j+1}_+, \  j=1,\ldots,n,\\
& X=\sum_{1\leq i<j\leq s(n,r)}M^{ij},\ M^{ij}\in\mathcal{S}^{s(n,r)},\ 1\leq i<j\leq s(n,r), \\
& M^{ij}_{ab}=0, \ \forall (a,b)\notin\{(i,i),(i,j),(j,i),(j,j)\}, \ 1\leq i<j\leq s(n,r),\\
&\left\|\left(
  \begin{array}{c}
    2M^{ij}_{ij} \\
    M^{ij}_{ii} - M^{ij}_{jj} \\
  \end{array}
\right)\right\|\leq M^{ij}_{ii}+M^{ij}_{jj}, \ 1\leq i<j\leq s(n,r),\\
&c_{\p,\q}\geq0, \ (\p,\q) \in N_k   \ \bigg\}.
\end{align*}

Its Lagrangian dual problem is the following SDP-SOCP problem:
\begin{align}\label{LSDkdr}
\inf\limits_{\substack{\by\in\R^{s(n,2\dd)}}}  \quad & \sum_{\ba\in\N^n_{2\dd}}(f_0)_\ba y_\ba  \tag{$\widetilde{\rm SD}_k^{\bfd,r}$} \\
 {\rm s.t. } \qquad \ &  \sum_{\ba\in\N^n_{2\dd}}(h_{\p,\q})_\ba y_\ba\geq0, \  (\p,\q)\in N_k, \nonumber \\
&\sum_{\ba\in\N^n_{2\dd}}y_\ba (\lceil x_j\rceil_{d_j}\lceil x_j\rceil_{d_j}^T)_\ba\succeq0, \ j=1,\ldots,n, \nonumber \\
&\left\|\left(
  \begin{array}{c}
    2(\M_r(\by))_{ij} \\
    (\M_r(\by))_{ii} - (\M_r(\by))_{jj} \\
  \end{array}
\right)\right\|\leq (\M_r(\by))_{ii}+(\M_r(\by))_{jj},  \nonumber \\
&1\leq i<j\leq s(n,r), \  y_{\bze}=1.\nonumber
\end{align}

The following theorem shows that both the sequences of the optimal values associated with the problems \eqref{SDkdr} and \eqref{LSDkdr} converge asymptotically to the global minimum of the primal problem~\eqref{P}.
For convenience, we denote by ${\rm val}\eqref{P},$ ${\rm val}\eqref{SDkdr}$ and ${\rm val}\eqref{LSDkdr}$ the optimal values of the problems~\eqref{P}, \eqref{SDkdr} and \eqref{LSDkdr}, respectively.

\begin{theorem}\label{thm3}
Let $\bfd\in\N^n$ and $r\in\N$ be fixed.
If {\rm Assumption~\ref{assumption1}} holds$,$ then
\begin{align*}
\lim_{k\to\infty}{\rm val}\eqref{SDkdr}=\lim_{k\to\infty}{\rm val}\eqref{LSDkdr}={\rm val}\eqref{P}.
\end{align*}
\end{theorem}

\begin{proof}
For all $k\in \N,$ the inequality {\rm val} \eqref{SDkdr} $\leq$ {\rm val} \eqref{LSDkdr} always holds by weak duality for conic programming \cite{Vandenberghe1995}, and the inequality $\lim_{k\to\infty}{\rm val}\eqref{SDkdr} \geq {\rm val}\eqref{P}$ can be proved by an argument similar to that in \cite[Theorem~3.1]{Jiao2026}.
In order to complete the proof, it suffices to show that ${\rm val}\eqref{LSDkdr}\leq{\rm val}\eqref{P}$ for all $k.$

Now, fix $k \in \N$.
Let $\bx$ be an arbitrary feasible solution of the problem \eqref{P}, and let $\by:=(y_\ba)=(\bx^\ba)\in\R^{s(n,2\dd)}.$
Then we have $y_{\bze}=1$ and
\begin{align*}
0\leq h_{\p,\q}(\bx)= \sum_{\ba\in\N^n_{2\dd}}(h_{\p,\q})_\ba \bx^\ba=\sum_{\ba\in\N^n_{2\dd}}(h_{\p,\q})_\ba y_\ba
\end{align*}
for all $(\p,\q)\in N_k.$
In addition, from the definition of the moment matrix we see that
\begin{equation*}
	\M_{r}(\by)=\sum_{\ba\in \N^n_{2\dd}}y_\ba (\lceil\bx\rceil_r\lceil\bx\rceil_r^T)_\ba=\lceil\bx\rceil_r\lceil\bx\rceil_r^T \succeq 0 ,
\end{equation*}
which implies that
\begin{align*}
\left\|\left(
  \begin{array}{c}
    2(\M_r(\by))_{ij} \\
    (\M_r(\by))_{ii} - (\M_r(\by))_{jj} \\
  \end{array}
\right)\right\|\leq (\M_r(\by))_{ii}+(\M_r(\by))_{jj}, \ 1\leq i<j\leq s(n,r).
\end{align*}
On the other hand,
\begin{align*}
\lceil x_j\rceil_{d_j}\lceil x_j\rceil_{d_j}^T=\sum_{\ba\in\N^n_{2\dd}}\bx^{\ba} \left(\lceil x_j\rceil_{d_j}\lceil x_j\rceil_{d_j}^T\right)_\ba=\sum_{\ba\in\N^n_{2\dd}}y_{\ba} \left(\lceil x_j\rceil_{d_j}\lceil x_j\rceil_{d_j}^T\right)_\ba\succeq0,\ j=1,\ldots,n.
\end{align*}
Therefore, $\by\in\R^{s(n,2\dd)}$ is a feasible solution of the problem  $\eqref{LSDkdr},$ and so,
\begin{align*}
f_{0}(\bx)= \sum_{\ba\in\N^n_{2\dd}}(f_{0})_\ba \bx^\ba=\sum_{\ba\in\N^n_{2\dd}}(f_{0})_\ba y_\ba\geq {\rm val}\eqref{LSDkdr}.
\end{align*}
Since $\bx$ is an arbitrary feasible solution of the problem~\eqref{P}, it holds that ${\rm val}\eqref{LSDkdr}\leq {\rm val}\eqref{P}$ for all $k$, and the desired result follows.
\end{proof}

\begin{remark}
{\rm
The proof of Theorem~\ref{thm3} does not rely on a strict SPLD degree separation in the polynomial data.
Indeed, the fixed-degree cone
\begin{equation}\label{cone}
\widetilde{\Sigma}[\bx]_r+\sum_{j=1}^n\Sigma[x_j]_{d_j}
\end{equation}
contains the zero polynomial.
Hence, the Krivine--Stengle certificate is included in the proposed hierarchy as the special case $\sigma=0$ and $\sigma_j=0$, $j=1,\ldots,n$.
Consequently, the asymptotic convergence result remains valid for general polynomial data satisfying Assumption~\ref{assumption1}.

The SPLD structure instead motivates the particular choice of the cone \eqref{cone}.
When the high-degree terms are separable and the nonseparable terms have lower degree, the univariate SOS components are designed to handle the former, while the degree of the multivariate SDSOS component can be chosen according to the lower-degree nonseparable part.
}\end{remark}

We next provide a sufficient condition that guarantees finite convergence of the dual BSDSOS-SPLD hierarchy defined by \eqref{LSDkdr} and allows the extraction of at least one global
minimizer.
The result may be proved by following a similar argument of \cite[Theorem 3.2]{Jiao2026} (see also \cite[Lemma 1]{Lasserre2017}).
Nevertheless, for the convenience of the reader, we provide a detailed proof to make the paper self-contained.

\begin{theorem}\label{thm4}
Let $\bfd\in\N^n$ and $r\in\N$ be fixed$,$ and let $\bar{s}\in \N$ and $\bar{l}\in\N$ be the smallest integers such that
\[
2\bar{s}\geq \max_{0\leq i\leq m}\left\{\deg (s_i)\right\} \quad {\textrm{and}} \quad 2\bar{l}\geq \max_{0\leq i\leq m}\left\{\deg (l_i)\right\},
\]
respectively.
Assume that $\overline \by \in\R^{s(n,2\dd)}$ is an optimal solution to the problem \eqref{LSDkdr}.
If
\begin{equation}\label{rank_con}
\max_{1\leq j\leq n}\left\{{\rm rank}\left(\sum_{\ba\in\N^n_{2\dd}}\overline y_\ba (\lceil x_j\rceil_{\bar{s}}\lceil x_j\rceil_{\bar{s}}^T)_\ba\right)\right\}={\rm rank}\left(\sum_{\ba\in\N^n_{2\dd}}\overline y_\ba (\lceil \bx\rceil_{\bar{l}}{\lceil \bx\rceil_{\bar{l}}}^T)_\ba\right)=1,
\end{equation}
then ${\rm val}\eqref{LSDkdr}={\rm val}\eqref{P}$ and $\overline \bx:=(\overline y_\ba)_{|\ba|=1}\in\R^n$ is an optimal solution to the problem \eqref{P}.
\end{theorem}

\begin{proof}
Assume that the condition \eqref{rank_con} holds.
Then there exists $\overline \bx\in\R^n$ such that
\begin{align*}
&\sum_{\ba\in\N^n_{2\dd}}\overline y_\ba (\lceil x_j\rceil_{\bar{s}}\lceil x_j\rceil_{\bar{s}}^T)_\ba=\lceil \overline x_j\rceil_{\bar{s}}\lceil \overline x_j\rceil_{\bar{s}}^T, \ j=1,\ldots,n,\\
&\sum_{\ba\in\N^n_{2\dd}}\overline y_\ba (\lceil \bx\rceil_{\bar{l}}\lceil \bx\rceil_{\bar{l}}^T)_\ba=\lceil \overline \bx\rceil_{\bar{l}}\lceil {\overline \bx\rceil_{\bar{l}}}^T.
\end{align*}
These lead to the facts that $\overline y_\ba=(\overline \bx)^\ba$ for all $\ba\in\N^n_{2\bar{l}}$ and $\overline y_{\ba_j}=(\overline x_j)^{\alpha_j}$ for all $\alpha_j\in\N_{2\bar{s}}.$
Here, $\ba_j \in \N_{2\bar{s}}^n$ is a multi-index such that the $j$-th component is $\alpha_j$ and all other components are equal to $0.$

Since $\overline \by$ is a feasible solution of the problem \eqref{LSDkdr}, we have
\begin{equation}\label{rel2_thm2}
\sum_{\ba\in\N^n_{2\dd}}(h_{\p,\q})_\ba \overline y_\ba\geq0, \  (\p,\q)\in N_k.
\end{equation}
Note that, for each $i=0,1,\ldots,m,$ $f_i(\bx)=s_i(\bx)+l_i(\bx),$ $2\bar{s}\geq \deg(s_i),$ and $2\bar{l}\geq \deg (l_i).$
By the choices of $d_j$ and $r$, and by the definition of $d_{\min}$ in \eqref{dmax}, we have
\begin{equation*}
d_{\min}\geq \max_{1\leq j\leq n}d_j\geq\bar s \quad\text{and}\quad d_{\min}\geq r\geq\bar l.
\end{equation*}
Now, let $i\in\{1,\ldots,m\}$ be any fixed and choose $(\p, \q) \in N_{k}$ such that $p_j = 1$ if $j = i,$ otherwise $p_j = 0,$ and $\q = \bze.$
Then, from \eqref{rel2_thm2}, we have
\begin{align*}
0\leq \sum_{\ba\in\N^n_{2\dd}}(f_i)_\ba \overline y_\ba=\sum_{\ba\in\N^n_{2\bar{s}}}(s_i)_\ba \overline \bx^\ba+\sum_{\ba\in\N^n_{2\bar{l}}}(l_i)_\ba \overline \bx^\ba=s_i(\overline\bx)+l_i(\overline\bx)=f_i(\overline\bx).
\end{align*}
Similarly, if we choose $(\p, \q) \in N_{k}$ such that $\p = \bze$ and  $q_j=1$ if $j = i,$ otherwise $q_j = 0,$ then, from \eqref{rel2_thm2}, we have
\begin{align*}
0\leq 1-\sum_{\ba\in\N^n_{2\dd}}(f_i)_\ba \overline y_\ba=1-f_i(\overline\bx).
\end{align*}
Since $i\in\{1,\ldots,m\}$ is arbitrary, $\overline \bx$ is a feasible solution of the problem \eqref{P}.
Hence we have
\begin{align*}
{\rm val}\eqref{P}\geq {\rm val}\eqref{LSDkdr}=\sum_{\ba\in\N^n_{2\dd}}(f_0)_\ba \overline y_\ba=f_0(\overline \bx)\geq{\rm val}\eqref{P},
\end{align*}
where the first inequality follows from the proof of Theorem~\ref{thm3}, while the last inequality follows from the feasibility of $\overline\bx$ for problem \eqref{P}.
Thus $\overline \bx=(\overline y_\ba)_{|\ba|=1}\in\R^n$ is an optimal solution to the problem \eqref{P}.
\end{proof}

\begin{remark}
{\rm
The extraction argument in Theorem~\ref{thm4} relies only on the decomposition $f_i=s_i+l_i$ and on the rank-one conditions imposed up to the degrees of the two components.
Hence, the argument remains valid even when the separable and nonseparable components have the same degree.
Nevertheless, this degree separation may offer a computational advantage when it allows the order of the full multivariate moment matrix to remain low.
In such a case, the high-degree separable terms are handled through univariate moment matrices, while the full multivariate moment matrix only needs to cover the lower-degree nonseparable terms.
}\end{remark}

We next establish strong duality between the problems \eqref{SDkdr} and \eqref{LSDkdr} for all sufficiently large relaxation orders.
\begin{proposition}\label{proposition2}
Let $\bfd\in\N^n$ and $r\in\N$ be fixed.
Suppose that {\rm Assumption~\ref{assumption1}} holds and that there exists $\widehat\bx\in\R^n$ such that
\begin{equation}\label{prop2eqn1}
0<f_i(\widehat\bx)<1, \ i=1,\ldots,m.
\end{equation}
Then, there exists $k_0\in\N$ such that for every $k\geq k_0,$
\begin{equation}\label{prop2eqn2}
\max\eqref{SDkdr}=\inf\eqref{LSDkdr}.
\end{equation}
\end{proposition}

\begin{proof}
Since the polynomials $f_i$, $i=1,\ldots,m$, are continuous, condition~\eqref{prop2eqn1} implies that there exists $\epsilon>0$ such that  
\begin{equation}\label{prop2eqn3}
0<f_i(\bx)<1, \quad \bx\in B, \quad i=1,\ldots,m,
\end{equation}
where $B:=\left\{\bx\in\R^n \colon \|\bx-\widehat\bx\|\leq\epsilon\right\}.$

Let $\nu$ be the normalized Lebesgue measure on $B$, so that $\nu(B)=1$.
Fix $k\geq1$, and let $\dd$ be defined as in~\eqref{dmax}.
Define $\widehat\by\in\R^{s(n,2\dd)}$ by
\begin{equation}\label{prop2eqn4}
\widehat y_\ba:=\int_B\bx^\ba d\nu(\bx),\quad \ba\in\N^n_{2\dd}.
\end{equation}
Clearly, $\widehat y_{\bze}=1$.

For every $(\p,\q)\in N_k$, condition~\eqref{prop2eqn3} implies that
\begin{equation*}
h_{\p,\q}(\bx)=\prod_{i=1}^mf_i(\bx)^{p_i}\left(1-f_i(\bx)\right)^{q_i}>0
\end{equation*}
for all $\bx\in B$. 
Hence,
\begin{equation}\label{prop2eqn5}
\sum_{\ba\in\N^n_{2\dd}}(h_{\p,\q})_\ba\widehat y_\ba=\int_B h_{\p,\q}(\bx) d\nu(\bx)>0, \quad (\p,\q)\in N_k.
\end{equation}

Next, for each $j=1,\ldots,n$, define
\begin{equation*}
Y_j:=\sum_{\ba\in\N^n_{2\dd}}\widehat y_\ba\left(\lceil x_j\rceil_{d_j}\lceil x_j\rceil_{d_j}^T\right)_\ba.
\end{equation*}
By~\eqref{prop2eqn4},
\begin{equation*}
Y_j=\int_B\lceil x_j\rceil_{d_j}\lceil x_j\rceil_{d_j}^T d\nu(\bx).
\end{equation*}
Let $\bz\in\R^{d_j+1}\setminus\{\bze\}$. 
Then
\begin{equation*}
\bz^TY_j\bz=\int_B\left(\bz^T\lceil x_j\rceil_{d_j}\right)^2
d\nu(\bx).
\end{equation*}
The polynomial
\begin{equation*}
\bx\longmapsto \bz^T\lceil x_j\rceil_{d_j}
\end{equation*}
is nonzero. 
Since $B$ has nonempty interior, this polynomial cannot vanish identically on $B$. 
Hence, 
\begin{equation*}
\bz^TY_j\bz>0.
\end{equation*}
It follows that 
\begin{equation}\label{prop2eqn6}
Y_j\succ0, \ j=1,\ldots,n.
\end{equation}

Similarly, by the definition of the moment matrix and \eqref{prop2eqn4}, we have
\begin{equation*}
\M_r(\widehat\by)=\int_B\lceil\bx\rceil_r\lceil\bx\rceil_r^T d\nu(\bx).
\end{equation*}
For every nonzero $\bz\in\R^{s(n,r)}$, the polynomial
\begin{equation*}
\bx\longmapsto \bz^T\lceil\bx\rceil_r
\end{equation*}
is nonzero and cannot vanish identically on $B$.
Consequently,
\begin{equation}\label{prop2eqn7}
\M_r(\widehat\by)\succ0.
\end{equation}

We now verify that the second-order cone constraints in \eqref{LSDkdr} are satisfied strictly. 
Fix $1\leq i<j\leq s(n,r).$
It follows from~\eqref{prop2eqn7} that
\begin{equation*}
\begin{pmatrix}
(\M_r(\widehat\by))_{ii} & (\M_r(\widehat\by))_{ij}\\
(\M_r(\widehat\by))_{ij} & (\M_r(\widehat\by))_{jj}
\end{pmatrix}
\succ0.
\end{equation*}
In particular, $(\M_r(\widehat\by))_{ii}+(\M_r(\widehat\by))_{jj}>0$ and $(\M_r(\widehat\by))_{ii}(\M_r(\widehat\by))_{jj}-\left((\M_r(\widehat\by))_{ij}\right)^2>0$.
Moreover,
\begin{align*}
\,&\left((\M_r(\widehat\by))_{ii}+(\M_r(\widehat\by))_{jj}\right)^2-
\left\|\begin{pmatrix}
2(\M_r(\widehat\by))_{ij}\\
(\M_r(\widehat\by))_{ii}-(\M_r(\widehat\by))_{jj}
\end{pmatrix}
\right\|^2\\
=\,&\left((\M_r(\widehat\by))_{ii}+(\M_r(\widehat\by))_{jj}\right)^2-\left\{4\left((\M_r(\widehat\by))_{ij}\right)^2+\left((\M_r(\widehat\by))_{ii}-(\M_r(\widehat\by))_{jj}\right)^2\right\}\\
=\,&4\left\{(\M_r(\widehat\by))_{ii}(\M_r(\widehat\by))_{jj}-\left((\M_r(\widehat\by))_{ij}\right)^2\right\}>0.
\end{align*}
Therefore,
\begin{equation}\label{prop2eqn8}
\left\|\begin{pmatrix}
2(\M_r(\widehat\by))_{ij}\\
(\M_r(\widehat\by))_{ii}-(\M_r(\widehat\by))_{jj}
\end{pmatrix}\right\|
<(\M_r(\widehat\by))_{ii}+(\M_r(\widehat\by))_{jj}.
\end{equation}
Thus, \eqref{prop2eqn5}, \eqref{prop2eqn6}, and \eqref{prop2eqn8} yield that $\widehat\by$ is a strictly feasible solution of \eqref{LSDkdr}. 
Since $k\geq1$ was arbitrary, \eqref{LSDkdr} satisfies Slater's condition for every $k\geq1$.

It remains to show that \eqref{SDkdr} is feasible for all sufficiently large $k$. 
Let $f^*:={\rm val}\eqref{P}$.
Since $\mathbf{K}$ is nonempty and compact, $f^*$ is finite.
Fix $\eta>0$. 
Then $f_0-f^*+\eta>0$ on $\mathbf{K}$. 
Hence, by Lemma~\ref{KS_positivity}, there exist finitely many nonnegative coefficients $c_{\p,\q}$ such that
\begin{equation}\label{prop2eqn9}
f_0-f^*+\eta-\sum_{\p,\q\in\N^m_0}c_{\p,\q}h_{\p,\q}\equiv0.
\end{equation}
Since only finitely many coefficients in \eqref{prop2eqn9} are nonzero, there exists an integer $k_0\geq1$ such that $c_{\p,\q}=0$ whenever $|\p|+|\q|>k_0$.
Therefore, for every $k\geq k_0$,
\begin{equation}\label{prop2eqn10}
f_0-\sum_{(\p,\q)\in N_k}c_{\p,\q}h_{\p,\q}-(f^*-\eta)\equiv0.
\end{equation}
Therefore, for every $k\geq k_0$, \eqref{prop2eqn10} shows that \eqref{SDkdr} is feasible with
\begin{equation*}
\mu=f^*-\eta, \quad \sigma=0, \quad \sigma_j=0, \ j=1,\ldots,n,
\end{equation*}
and the coefficients $(c_{\p,\q})_{(\p,\q)\in N_k}$ given in \eqref{prop2eqn9}.
Thus, \eqref{SDkdr} is feasible for every $k\geq k_0$.

For every such $k$, the dual problem \eqref{LSDkdr} satisfies Slater's condition. 
Hence, by the strong duality theorem for conic programming, the supremum in \eqref{SDkdr} is attained and
\begin{equation*}
\max\eqref{SDkdr}=\inf\eqref{LSDkdr};
\end{equation*}
see, e.g., \cite{Vandenberghe1995}. 
This proves \eqref{prop2eqn2}.
\end{proof}

Combining Theorem~\ref{thm4} with Proposition~\ref{proposition2}, we obtain a sufficient condition for the finite convergence of both the primal and dual BSDSOS-SPLD hierarchies defined by
\eqref{SDkdr} and \eqref{LSDkdr}, respectively.
\begin{corollary}
Suppose that the assumptions of Proposition~\ref{proposition2} hold. 
Assume that, at some relaxation order $k_0\in\N$, the problem \eqref{LSDkdr} has an optimal solution $\overline\by\in\R^{s(n,2\dd)}$ satisfying the rank condition~\eqref{rank_con}. 
Then
\begin{equation*}
\overline\bx:=(\overline y_\ba)_{|\ba|=1}\in\R^n
\end{equation*}
is an optimal solution to the problem~\eqref{P}. 
Moreover, there exists $k_1\geq k_0$ such that, for every $k\geq k_1$,
\begin{equation*}
\max\eqref{SDkdr}=\min\eqref{LSDkdr}=\min\eqref{P}.
\end{equation*}
\end{corollary}

\begin{proof}
By Theorem~\ref{thm4}, at relaxation order $k_0$, 
\begin{equation*}
{\rm val}\eqref{LSDkdr}={\rm val}\eqref{P},
\end{equation*}
and $\overline\bx$ is an optimal solution to the problem \eqref{P}. 
Since the sequence of optimal values of problem \eqref{LSDkdr} is nondecreasing and bounded above by ${\rm val}\eqref{P}$, this equality holds for every $k\geq k_0$.

By Proposition~\ref{proposition2}, there exists $\widehat k\in\N$ such that
\begin{equation*}
\max\eqref{SDkdr}=\inf\eqref{LSDkdr}
\end{equation*}
for every $k\geq\widehat k$. 
Letting $k_1:=\max\{k_0,\widehat k\}$ gives the desired result.
\end{proof}

\subsection{Numerical experiments}\label{sect:6}
Now, we present numerical experiments to illustrate the proposed BSDSOS-SPLD hierarchy.
All the numerical examples are conducted on a computer with a 2.20 GHz Intel(R) Core(TM) Ultra 7 258V CPU and 32.0 GB RAM, equipped with MATLAB R2021b.

We select and solve the following nonconvex polynomial optimization problems (see, e.g., \cite{Chuong2019,Jiao2026,Lasserre2017}):
\begin{align}\label{Pnd}
\min_{\bx\in\mathbb{R}^{N}} \quad & f_{N,D}(\bx)     \tag{${\rm P}_{N,D}$}\\
\mathrm{s.t.}\quad & 0\leq G_{N,D}^{j}(\bx)\leq 1, \ j=1,\ldots,5,\notag\\
& 0\leq x_i\leq 1, \ i=1,\ldots,N.\notag
\end{align}
Here, $N$ and $D$ are positive even integers.
The objective polynomial is
\begin{equation*}
f_{N,D}(\bx)=\sum_{i=1}^{N}(-1)^{i+1}x_i^{D}+x_1-x_2,
\end{equation*}
and the five constraint polynomials are defined by
\begin{align*}
G_{N,D}^{1}(\bx)&=\sum_{p=1}^{N/2}\left(2x_{2p-1}^{D}+3x_{2p}^{2}+2x_{2p-1}x_{2p}\right), & G_{N,D}^{2}(\bx)&=\sum_{p=1}^{N/2} \left( 3x_{2p-1}^{2}+2x_{2p}^{2}-4x_{2p-1}x_{2p} \right),\\
G_{N,D}^{3}(\bx)&=\sum_{p=1}^{N/2}\left(x_{2p-1}^{2}+6x_{2p}^{2}-4x_{2p-1}x_{2p}\right), &G_{N,D}^{4}(\bx)&=\sum_{p=1}^{N/2}\left(x_{2p-1}^{2}+4x_{2p}^{D}-3x_{2p-1}x_{2p}\right),\\
G_{N,D}^{5}(\bx)&=\sum_{p=1}^{N/2}\left(2x_{2p-1}^{2}+5x_{2p}^{2}+3x_{2p-1}x_{2p}\right).
\end{align*}

The results are presented in Table \ref{table1}. They provide the following information:
\begin{itemize}
\item The pair $(d,k)$ denotes the fixed SDSOS order $d$ and the relaxation order $k$ in the BSDSOS hierarchy.
Similarly, $(d_0,r,k)$ denotes the fixed univariate SOS order $d_0$, the fixed multivariate SDSOS order $r$, and the relaxation order $k$ in the BSDSOS-SPLD hierarchy.
\item The abbreviation $\verb"Opt"$ denotes the optimal value obtained by the BSDSOS hierarchy and BSDSOS-SPLD hierarchy at order $k.$
\item $\verb"Time"$ stands for the time in seconds including the time for variable declaration in YALMIP, SDP and SOCP formulations and solving, and verification of the optimality condition.
\item The abbreviations $\verb"ns"$ and $\verb"ss"$ refer to the number of second-order cone constraints in each hierarchy and the maximum size of univariate SDP constraints in BSDSOS-SPLD hierarchy, respectively.
\item The abbreviation $\verb"rnk"$ stands for the numerical rank of the moment matrix in BSDSOS hierarchy; and the abbreviation $\verb"max_rnk"$ stands for the maximum of the numerical ranks of the moment matrix $\M_r(\by)$ and all moment matrices $\M_{d_0}(\by^j)$ in BSDSOS-SPLD hierarchy\footnote{For each matrix $M$, let $\sigma_1(M)\geq \sigma_2(M)\geq \cdots \geq 0$ denote its singular values. We declare that $M$ has numerical rank one if $\sigma_2(M)/\sigma_1(M)\leq 10^{-4}$, or equivalently, if $\sigma_1(M)$ is at least $10^4$ times larger than $\sigma_2(M)$.}.
\end{itemize}

\begin{table}[t]
\caption{Comparison of BSDSOS and BSDSOS-SPLD for the test problems}
\label{table1}
\centering{\footnotesize
\begin{tabular}{cccccccccccc}
\toprule
\multicolumn{1}{l}{\textbf{Problem}} & \multicolumn{5}{l}{\textbf{BSDSOS} \cite{Chuong2019}}  & \multicolumn{5}{l}{\textbf{BSDSOS-SPLD}} \\
\cmidrule(rl){2-6} \cmidrule(rl){7-12}
                & $d,k$ & $\verb"Opt"$ & $\verb"Time"$ & $\verb"ns"$ & \verb"rnk" &$d_0,r,k$ & $\verb"Opt"$ &$\verb"Time"$ & $\verb"ns"$ & $\verb"ss"$ & \verb"max_rnk" \\
\toprule
$({\rm P}_{6,6})$       & $3,1$   & $-0.6963$  & $1.7$         & $3486$      & $27$       & $28,2,1$ & $-0.6274$    & $0.4$        & $378$       & $29$        & $4$         \\
                & $3,2$   & $-0.6676$  & $1.8$         & $3486$      & $45$       & $28,2,2$ & $-0.4161$    & $0.6$        & $378$       & $29$        & $4$         \\
                & $3,3$   & $-0.4129$  & $3.7$         & $3486$      & $1$        & $28,2,3$ & $-0.4129$    & $2.2$        & $378$       & $29$        & $1$         \\
\midrule
$({\rm P}_{8,6})$       & $3,1$   & $-0.6963$  & $6.7$        & $13530$     & $36$       & $28,2,1$ & $-0.6265$    & $0.7$        & $990$       & $29$        & $5$         \\
                & $3,2$   & $-0.6677$  & $7.1$        & $13530$     & $81$       & $28,2,2$ & $-0.4162$    & $0.9$        & $990$       & $29$        & $6$         \\
                & $3,3$   & $-0.4129$  & $11.9$        & $13530$     & $1$        & $28,2,3$ & $-0.4129$    & $5.2$       & $990$       & $29$        & $1$         \\
\midrule
$({\rm P}_{10,6})$& $3,1$   & $-0.6963$  & $22.7$        & $40755$     & $57$       & $28,2,1$ & $-0.6268$    & $1.5$        & $2145$      & $29$        & $6$         \\
                & $3,2$   & $-0.6677$  & $28.7$        & $40755$     & $141$      & $28,2,2$ & $-0.4163$    & $1.9$        & $2145$      & $29$        & $8$         \\
                & $3,3$   & $-0.4129$  & $36.2$        & $40755$     & $1$        & $28,2,3$ & $-0.4129$    & $13.6$       & $2145$      & $29$        & $1$         \\
\midrule
$({\rm P}_{12,6})$& $3,1$   & $-0.6963$  & $167.8$       & $103285$    & $85$       & $28,2,1$ & $-0.6290$    & $2.5$        & $4095$      & $29$        & $7$         \\
                & $3,2$   & $-0.6677$  & $170.7$       & $103285$    & $225$      & $28,2,2$ & $-0.4164$    & $3.5$        & $4095$      & $29$        & $10$         \\
                & $3,3$   & $-0.4129$  & $230.9$       & $103285$    & $1$        & $28,2,3$ & $-0.4129$    & $48.5$       & $4095$      & $29$        & $1$         \\
\midrule
$({\rm P}_{6,8})$ & $4,1$   & $-0.6963$  & $12.0$        & $21945$     & $46$       & $45,2,1$ & $-0.6084$    & $0.5$        & $378$       & $46$        & $4$         \\
                & $4,2$   & $-0.6677$  & $11.7$        & $21945$     & $82$       & $45,2,2$ & $-0.4131$    & $0.7$        & $378$       & $46$        & $4$         \\
                & $4,3$   & $-0.6597$  & $13.7$        & $21945$     & $45$       & $45,2,3$ & $-0.4090$    & $2.5$        & $378$       & $46$        & $1$         \\
\midrule
$({\rm P}_{8,8})$ & $4,1$   & $-0.6963$  & $239.3$       & $122265$    & $113$      & $45,2,1$ & $-0.6131$    & $0.8$        & $990$       & $46$        & $5$         \\
                & $4,2$   & $-0.6677$  & $234.4$       & $122265$    & $187$      & $45,2,2$ & $-0.4133$    & $1.2$        & $990$       & $46$        & $6$         \\
                & $4,3$   & $-0.6597$  & $164.6$       & $122265$    & $105$       & $45,2,3$ & $-0.4090$    & $8.1$       & $990$       & $46$        & $1$         \\
\midrule
$({\rm P}_{10,8})$& $4,1$   & $-0.6963$  & $3146.4$      & $500500$    & $222$      & $45,2,1$ & $-0.6132$    & $1.5$        & $2145$      & $46$        & $6$         \\
                & $4,2$   & $-0.6677$  & $2843.7$      & $500500$    & $332$      & $45,2,2$ & $-0.4133$    & $3.6$        & $2145$      & $46$        & $9$         \\
                & $4,3$   & $-0.6597$  & $4080.8$      & $500500$    & $221$      & $45,2,3$ & $-0.4090$    & $54.8$      & $2145$      & $46$        & $1$         \\
\midrule
$({\rm P}_{6,10})$& $5,1$   & $-0.6963$  & $125.6$       & $106491$    & $73$       & $66,2,1$ & $-0.6000$    & $0.7$        & $378$       & $67$        & $4$         \\
                & $5,2$   & $-0.6678$  & $163.5$       & $106491$    & $132$      & $66,2,2$ & $-0.4124$    & $1.0$        & $378$       & $67$        & $4$         \\
                & $5,3$   & $-0.6597$  & $209.9$       & $106491$    & $76$       & $66,2,3$ & $-0.4084$    & $2.9$        & $378$       & $67$        & $1$         \\
\midrule
$({\rm P}_{8,10})$& $5,1$   & $-0.6963$  & $7358.5$      & $827541$    & $152$      & $66,2,1$ & $-0.5990$    & $1.1$        & $990$       & $67$        & $5$         \\
                & $5,2$   & $-0.6678$  & $7186.8$      & $827541$    & $314$      & $66,2,2$ & $-0.4124$    & $2.1$        & $990$       & $67$        & $6$         \\
                & $5,3$   & $-0.6597$  & $7415.8$      & $827541$    & $199$      & $66,2,3$ & $-0.4084$    & $20.1$       & $990$       & $67$        & $1$         \\
\bottomrule\end{tabular}}
\end{table}

In these numerical experiments, we compare the standard BSDSOS hierarchy~\cite{Chuong2019} with the proposed BSDSOS-SPLD hierarchy on SPLD polynomial programs of varying dimensions and degrees.
For the test problem \eqref{Pnd}, we have $D_s=D$ and $D_l=2$, and hence $(P_{N,D})$ has an SPLD structure.
In each problem, the dimension $N\in\{6,8,10,12\}$ (number of variables) and polynomial degree $D\in\{6,8,10\}$ are even, and we solve the $k$th order relaxation for $k=1,2,3.$
All relaxations are solved in MATLAB via YALMIP \cite{Lofberg2004}, with MOSEK \cite{Mosek} as the conic solver.
The standard BSDSOS method constructs a single SDSOS certificate of degree $2d\geq D$ in all $N$ variables and thus requires a moment matrix of size $s(N,d)$.
By contrast, the BSDSOS-SPLD relaxation decomposes the certificate into a low degree multivariate part of degree $2r$ and univariate SOS polynomials of degree $2d_0.$

For convenience, the fixed SOS order $d_0$ is chosen so that $d_0+1$ is no larger than, or comparable to, the size $s(N,r)$ of the multivariate moment matrix.
This provides a simple rule for selecting $d_0$ without additional parameter tuning.
A smaller admissible value of $d_0$ may also be used and may still yield convergence at a low relaxation order.
Such a choice would reduce the sizes of the univariate semidefinite blocks.
In the present numerical experiments, however, we use the above convenient choice, although it may result in larger univariate moment matrices.

For the degree $6$ case ($D=6$), both hierarchies satisfy the rank-one condition at relaxation order $k=3,$ and thus converge to the same optimal value at that order.
In contrast, for $D\geq8$ the BSDSOS hierarchy does not satisfy the numerical rank-one condition within the tested relaxation orders $k\leq 3$, indicating that higher-order relaxations may be required to satisfy the numerical rank-one condition at such degrees, while the BSDSOS-SPLD attains $\verb"max_rnk" = 1$ at $k=3$ on all instances and typically produces lower bounds close to the final values by $k=2$, reflecting faster practical convergence.

In order to explain the conic complexity at a given level of the hierarchy, we first discuss the number of second-order cone (SOC) constraints induced by the order-$t$ moment matrix.
To this end, let $\M_t(\by)\in\mathcal{S}^{s(n,t)}$ be the moment matrix (indexed by all monomials in $n$ variables of total degree at most $t$).
For every pair $\{i,j\}$ with $1\leq i<j\leq s(n,t),$ the following SOC constraint is imposed on the entries of the moment matrix $\M_t(\by),$
\begin{equation}\label{rsocc}
\left\|\left(
  \begin{array}{c}
    2(\M_t(\by))_{ij} \\
    (\M_t(\by))_{ii} - (\M_t(\by))_{jj} \\
  \end{array}
\right)\right\|\leq (\M_t(\by))_{ii}+(\M_t(\by))_{jj},
\end{equation}
which is equivalent to
\begin{align*}
\left(
  \begin{array}{cc}
    (\M_t(\by))_{ii} & (\M_t(\by))_{ij} \\
    (\M_t(\by))_{ji} & (\M_t(\by))_{jj} \\
  \end{array}
\right)\succeq0.
\end{align*}
Hence \eqref{rsocc} enforces positive semidefiniteness for every principal $2\times2$ submatrix of $\M_t(\by)$.
Thus, the total number of SOC constraints is
\begin{equation*}
\binom{s(n,t)}{2}=\frac{s(n,t)(s(n,t)-1)}{2}.
\end{equation*}

\begin{remark}[{\bf A discussion on BSDSOS~\cite{Chuong2019} and BSDSOS-SPLD}]\label{rem2512}
{\rm
We now take the test problem \eqref{Pnd} as an example to compare the BSDSOS hierarchy~\cite{Chuong2019} with the proposed BSDSOS-SPLD hierarchy.
\begin{itemize}
\item The moment block in the BSDSOS hierarchy~\cite{Chuong2019} has size $s(N,d)$, and hence the number of SOC constraints is $\binom{s(N,d)}{2}$.
\item The BSDSOS-SPLD hierarchy uses the moment matrix $\M_r(\by)$ of degree $2r$, together with $N$ univariate SDP blocks for the high-degree separable terms.
The moment matrix has size $s(N,r)$ and therefore induces $\binom{s(N,r)}{2}$ SOC constraints, while the separable part is represented by $N$ univariate SDP blocks of size $d_0+1$.
\end{itemize}
Since $r<d$ in all the tested instances, the proposed hierarchy requires substantially fewer SOC constraints, as confirmed by the values of \texttt{ns} reported in Table~\ref{table1}.
It also exhibits shorter CPU times for all the tested instances.
}
\end{remark}

To examine Remark~\ref{rem2512} in more detail, we consider the cases $({\rm P}_{10,6}),$ $({\rm P}_{12,6})$ and $({\rm P}_{10,8}).$
For problem $({\rm P}_{10,6})$, the BSDSOS hierarchy with $d=3$ produces $\binom{286}{2}=40,755$ SOC constraints.
In the BSDSOS-SPLD hierarchy with $r=2$, $({\rm P}_{10,6})$ yields $\binom{s(10,2)}{2}=\binom{66}{2}=2,145$ SOC constraints, while the high-degree separable terms are represented by $N$ univariate
SDP blocks of size $d_0+1$.
Similarly, for $({\rm P}_{12,6})$ one obtains $s(12,3)=455$ and therefore $\binom{455}{2}=103,285$ SOC constraints under the BSDSOS hierarchy, but $s(12,2)=91$ and $\binom{91}{2}=4,095$ under the BSDSOS-SPLD hierarchy.
Under BSDSOS with $d=4$, the number of  SOC constraints for $({\rm P}_{10,8})$ equals $\binom{s(10,4)}{2}=\binom{1001}{2}=500,500$.
Under BSDSOS-SPLD with $r=2$, it equals $\binom{s(10,2)}{2}=\binom{66}{2}=2,145$, and the high–degree separable terms are handled by $N$ univariate SDP blocks of size $d_0+1=46.$
As the number of SOC constraints increases, interior-point methods require more memory and exhibit higher per-iteration factorization costs.
Consequently, the CPU times in Table~\ref{table1} increase, and the difference in CPU time between the two hierarchies grows as $D$ increases.
In particular, for $D\ge 8$ the BSDSOS hierarchy does not achieve the rank-one condition at $k=3.$
Since increasing the relaxation order expands the index set $N_k,$ i.e., $N_k\subset N_{k+1}$, the number of linear constraints
\begin{equation*}
\sum_{\ba\in\N^n_{2\dd}}(h_{\p,\q})_\ba y_\ba\geq0, \ (\p,\q)\in N_k
\end{equation*}
grows rapidly, which in turn increases the memory required to generate the constraints and the overall computational cost.
Nevertheless, the BSDSOS-SPLD hierarchy attains numerical rank one at $k=3$ for all tested instances.

Overall, as shown by the numerical experiments, the BSDSOS-SPLD hierarchy reduces conic complexity and CPU time compared with the BSDSOS hierarchy \cite{Chuong2019}.
Concretely, BSDSOS-SPLD uses a low-degree multivariate SDSOS polynomial (degree $2r$) together with $N$ univariate SOS polynomials for the separable high-degree terms (degree $2d_0$);
equivalently, the SDSOS term is represented by SOC constraints on all $2\times 2$ principal submatrices of the moment matrix $\M_r(\by),$ while the separable high-degree term appears as $N$ small univariate SDP blocks.
This hybrid SDP-SOCP design reduces the total number of conic constraints from hundreds of thousands to a few thousand, and enables higher degrees and dimensions to be solved within practical time and memory limits, whereas the BSDSOS hierarchy \cite{Chuong2019} requires much longer CPU time.

\section{Conclusions}\label{sect:conclusion}
In this paper, we have proposed a bounded degree SDSOS hierarchy for solving SPLD polynomial programs.
The asymptotic convergence of this hierarchy was proved and sufficient conditions for its finite convergence were provided.
Numerical experiments showed that our method outperformed the BSDSOS hierarchy proposed by Chuong et al.~\cite{Chuong2019}.

%\subsection*{Acknowledgments.}
%Liguo Jiao was partially supported by the Chinese National Natural Science Foundation under grant number 12471478.
%Jae Hyoung Lee was supported by the National Research Foundation of Korea (NRF) grant funded by the Korea government (MSIT) (NRF-2021R1C1C2004488).

\subsection*{Acknowledgments}
The authors would like to express their sincere gratitude to Professors J.-B. Lasserre and T.-S. Pham for helpful discussions during their visit to the Vietnam Institute for Advanced Study in Mathematics in October 2025. These discussions helped motivate and improve the present work.

%\subsection*{Fundings}
%Liguo Jiao was partially supported by the Chinese National Natural Science Foundation under grant number 12471478.
%Jae Hyoung Lee
%This work was supported by the National Research Foundation of Korea (NRF) grant funded by the Korea government (MSIT) (NRF-2021R1C1C2004488).

%\subsection*{Funding}
%Liguo Jiao was partially supported by the Chinese National Natural Science Foundation under grant numbers 12471478, 12371300, and was supported by the Education Department of Jilin Province (no. JJKH20241405KJ).
%Jae Hyoung Lee was supported by the National Research Foundation of Korea (NRF) grant funded by the Korea government (MSIT) (NRF-2021R1C1C2004488).

\subsection*{Availability of data and materials}
The codes and datasets generated in this study are available from the first author on a reasonable request.

\subsection*{Conflict of interest}
The authors declare that they have no conflict of interest.

\footnotesize
%\bibliography{BDH-4}

\end{document}